\documentclass[letterpaper, 10 pt, conference]{ieeeconf}

\IEEEoverridecommandlockouts                              
 \pdfoutput=1
\usepackage{amsmath, amssymb}
\usepackage{amsfonts}
\usepackage{graphicx}
\usepackage{enumerate}
\usepackage{caption}
\usepackage{ifthen}
\usepackage{multicol}
\usepackage{float}
\usepackage{cite}
\usepackage[usenames, dvipsnames]{color}
\usepackage[lofdepth,lotdepth]{subfig}
\graphicspath{{figures/}}
\usepackage{url}

\newtheorem{theorem}{Theorem}[section]
\newtheorem{lemma}[theorem]{Lemma}
\newtheorem{proposition}[theorem]{Proposition}
\newtheorem{corollary}[theorem]{Corollary}
\newtheorem{rem}{Remark}

\newboolean{showcomments}
\setboolean{showcomments}{true}

\newcommand{\henrik}[1]{\ifthenelse{\boolean{showcomments}}
{\textcolor{Blue}{Henrik says: #1}}{}}
\newcommand{\emma}[1]{\ifthenelse{\boolean{showcomments}}
{\textcolor{Green}{Emma says: #1}}{}}
\newcommand{\martin}[1]{\ifthenelse{\boolean{showcomments}}
{\textcolor{Skyblue}{Martin says: #1}}{}}
\newcommand{\john}[1]{\ifthenelse{\boolean{showcomments}}
{\textcolor{Red}{John says: #1}}{}}

\newboolean{shownew}
\setboolean{shownew}{true}
\newcommand{\newtext}[1]{\ifthenelse{\boolean{shownew}}
{\textcolor{Red}{#1}}{}}

\newcommand{\hn}{$\mathcal{H}_2$ }

\usepackage{array}
\newcolumntype{L}[1]{>{\raggedright\let\newline\\\arraybackslash\hspace{0pt}}m{#1}}
\newcolumntype{C}[1]{>{\centering\let\newline\\\arraybackslash\hspace{0pt}}m{#1}}
\newcolumntype{R}[1]{>{\raggedleft\let\newline\\\arraybackslash\hspace{0pt}}m{#1}}
\usepackage{subfig}

 \usepackage{tikz}
 \usetikzlibrary{decorations.pathmorphing,shadings}
 \usetikzlibrary{plotmarks}
 \usepackage{pgfplots}
\usepackage{tkz-graph}

\definecolor{gray3}{rgb}{0.75, 0.75, 0.75}
\definecolor{gray2}{rgb}{0.5, 0.5, 0.5}
\definecolor{gray1}{rgb}{0.25, 0.25, 0.25}
\definecolor{gray0}{rgb}{0.15, 0.15, 0.15}

\pgfplotscreateplotcyclelist{mygraylist}{%
black, densely dashed\\%
gray1\\%
gray3\\%
gray2\\%
gray2,dashed\\%
gray3,dashed\\%
}

\pgfplotscreateplotcyclelist{myothergraylist}{%
black, densely dashed\\%
gray1\\%
gray2\\%
gray3\\%
gray1, dashed\\%
gray3,dashed\\%
}

\usetikzlibrary{arrows}

\title{\LARGE \bf
Improving performance of droop-controlled microgrids through distributed PI-control}

\author{Emma Tegling, Martin Andreasson, John W. Simpson-Porco and Henrik Sandberg %
\thanks{E. Tegling, M. Andreasson and H. Sandberg are with the School of Electrical Engineering and the ACCESS Linnaeus Center, KTH Royal Institute of Technology, SE-100 44 Stockholm, Sweden {(\tt tegling, mandreas, hsan@kth.se)}. J. W. Simpson-Porco is with the Department of Electrical and Computer Engineering, University of Waterloo, ON, Canada {(\tt  jwsimpson@uwaterloo.ca)}.   }  
 }%

\begin{document}
\maketitle
\thispagestyle{empty}
\pagestyle{empty}

\begin{abstract}
This paper investigates transient performance of inverter-based microgrids in terms of the resistive power losses incurred in regulating frequency under persistent stochastic disturbances. We model the inverters as second-order oscillators and compare two algorithms for frequency regulation: the standard frequency droop controller and a distributed proportional-integral (PI) controller. The transient power losses can be quantified using an input-output \hn norm. We show that the distributed PI-controller, which has previously been proposed for secondary frequency control (the elimination of static errors), also has the potential to significantly improve performance by reducing transient power losses. This loss reduction is shown to be larger in a loosely interconnected network than in a highly interconnected one, whereas losses do not depend on connectivity if standard droop control is employed. 
%
%
Moreover, our results indicate that there is an optimal tuning of the distributed PI-controller for loss reduction.
Overall, our results provide an additional argument in favor of  distributed algorithms for secondary frequency control in microgrids. 
\end{abstract}

\section{Introduction}
\label{sec:intro}
Driven by environmental concerns and several economic factors, the electric power system is moving from a centralized generation paradigm towards a more distributed one. Local, small-scale  generation resources are expected to become prevalent in future power networks, as the penetration of renewable energy sources increases \cite{smart_report2012, Farhangi2010}. The \emph{microgrid} concept has gained popularity as a key strategy to facilitate this transition \cite{Markvart2006, Lasseter2002}. Microgrids are networks composed of distributed generation (DG) units, loads and energy storage elements which can either connect to a larger power grid, or operate independently from it, in ``islanded'' mode. 

The DG units within the microgrid are typically interfaced with the AC network via DC/AC or AC/AC power converters, or \emph{inverters}. The network's stability, synchronization and power balance depend on control actions taken in these inverters \cite{Lopes2006, ZhongHornik}. The standard control scheme employed to stabilize the system and achieve active power sharing, i.e., a desired steady-state distribution of power injections of inverter units, is \emph{droop control}, effectively a decentralized proportional controller. While droop control, under reasonable conditions (see e.g. \cite{SimpsonPorco2013}), is successful at stabilizing the network, it typically causes the steady-state network frequency to deviate from its nominal value \cite{Lopes2006}.

This deficiency motivates so-called secondary control, the goal of which is to eliminate the static error. In order to achieve this goal, control architectures with various degrees of centralization have been proposed. Unless carefully tuned, completely decentralized secondary controllers destroy the power sharing properties established by droop control, and may lead to a violation of generation constraints \cite{Andreasson2014ACC,Dorfler2014}. Conversely, centralized control requires a dense communication architecture and conflicts with the microgrid paradigm of autonomous management and scalability. This has motivated the development of distributed control algorithms which simultaneously eliminate frequency errors and maintain the optimality properties of droop control \cite{SimpsonPorco2013, SimpsonPorco2015,Andreasson2014ACC, ZhaoMallada2015}.

In this paper, we study one such control algorithm, which builds on the droop controller and combines it with integral control and distributed averaging algorithms. Stability and power sharing properties of such distributed PI-controllers have been studied in \cite{SimpsonPorco2013, SimpsonPorco2015} for first-order inverter models, and in \cite{Andreasson2014ACC,Andreasson2014TAC} for synchronous generator networks represented by second-order oscillators. Here, we model an inverter-based network with second-order dynamics. However, the novel aspect of the present paper lies in the analysis of transient performance, not in stability analysis.

We consider performance in terms of the resistive power losses incurred in regulating frequency under persistent small disturbances caused by, for example, variations in generation and loads. These losses are associated with power flows that arise from fluctuating phase angle differences, and can be regarded as a measure of control effort. The total transient power losses can be quantified through the \hn norm of an input-output system describing the coupled inverter dynamics, with an appropriately defined output. 

$\mathcal{H}_2-$based performance bounds have previously been used in \cite{Bamieh2012} to derive fundamental performance limits for large-scale vehicular platoons and consensus networks with spatial invariance. By exploiting spatial invariance properties, similar bounds for voltage regulation in DC power networks were derived in \cite{Rinehart2011}. For general coupled oscillator networks, robustness with respect to disturbances was studied in \cite{SiamiMotee2013} and methods to reduce inter-nodal interactions due to disturbances were proposed by \cite{Grunberg2014}. 

The present work adheres to \cite{BamiehGayme2013, Tegling2014}, where performance in terms of power losses was evaluated for synchronous generator networks, and to \cite{Tegling2015}, where droop-controlled microgrids with variable voltage dynamics were studied. These works showed that under the assumption of uniform generator parameters, the losses associated with frequency synchronization will scale unboundedly with the network size, but not depend on the network's connectivity. 
While this scaling of losses with network size appears to be a fundamental performance limit (cf. \cite{Bamieh2012}), the main result of this work is that these losses can be significantly decreased by applying distributed secondary PI-control. Surprisingly, we also find that the performance improvement over droop control is larger in a sparsely connected network than in a highly connected one. This stands in contrast to synchronization results in complex networks and power systems, which instead predict that densely connected networks are easier to synchronize \cite{Pecora1998, Schiffer2014}, are more coherent \cite{Bamieh2012}, and display faster rates of convergence \cite{Tang2011}. Our result therefore indicates that there is a fundamental trade-off between network coherency and transient power losses, in that additional power lines which strengthen synchronization also incur additional losses. 

Moreover, we find that there is an optimal tuning for the distributed PI-controller which minimizes the transient resistive losses. Numerically, we find that the optimal gain for the distributed averaging in the controller is often quite small, indicating that only only low-gain distributed feedback between controllers is needed to optimize transient performance. 



The remainder of this paper is organized as follows. We introduce the models for the inverters and the control strategies in Section~\ref{sec:setup}. In Section~\ref{sec:eval}, we evaluate performance and discuss network topology dependencies. In Section~\ref{sec:analysis}, optimal tuning of the distributed PI-controller is discussed, before we conclude in Section~\ref{sec:discussion}.  

\section{Problem setup}
\label{sec:setup}
Consider a network $\mathcal{G} = \{\mathcal{V}, \mathcal{E}\}$, where $\mathcal{V} = \{1, \ldots, N\}$ is the set of nodes and $\mathcal{E} = \{e_{ij}\}$ represents the set of edges, or network lines. Each network line is represented by a constant (complex) admittance $y_{ij} = g_{ij} - \mathbf{j}b_{ij}$, {where $g_{ij}, b_{ij} > 0$}. Throughout this paper, we will assume a Kron-reduced network model (see e.g. \cite{Varaiya1985, Nishikawa2015}), {where the reduction procedure eliminates the constant-impedance loads and absorbs their effects into} the network lines $\mathcal{E}$, along with any phase-shifting transformers. Consequently, every node $i \in \mathcal{V}$ represents a generation unit with a power inverter as its grid interface. Each node has an associated phase angle $\theta_i$ and voltage magnitude $|V_i|$. 

In this paper, we will focus on how frequency control impacts the performance in terms of resistive losses; see \cite{Tegling2015} for the impact of voltage droop control. 


\subsection{Inverter and droop control model}
\label{sec:model}
We first introduce the model for standard frequency droop controller. We assume the inverters at nodes $i \in \mathcal{V}$ to be AC voltage sources, whose frequency output can be regulated according to:
\begin{equation}
\label{eq:control}
\dot{\theta}_i = u_i,
\end{equation}
where $u_i$ is the control signal. The droop controller balances the active power demand through simple proportional control
\begin{equation}
\label{eq:droop}
u_i = \omega^* - m_i (\hat{P}_i - P_i^*), 
\end{equation}
where the controller gain $m_i>0$ is called the \emph{droop coefficient}, $\omega^*$ and $P^*_i$ are the frequency and active power setpoints, and $\hat{P}_i$ is the measured active power. Following \cite{Schiffer2014}, we assume measurement delay dynamics where $\hat{P}_i$ is measured and processed through a low-pass filter as
\begin{equation}
\label{eq:filter}
\tau_i \dot{\hat{P}}_i  = -\hat{P}_i +P_i\,,
\end{equation}
where $\tau_i>0$ is the time constant of the filter and $P_i$ is the actual power injection at node $i$ (Section~\ref{sec:powerflow}).

Now, we substitute \eqref{eq:droop} into \eqref{eq:control} and introduce the inverter frequency $\omega_i = \dot{\theta}_i$ to obtain
\begin{equation}
\label{eq:freq}
\omega_i = \omega^*- m_i(\hat{P}_i - P_i^*).
\end{equation}
Taking the time derivative of \eqref{eq:freq} gives $\dot{\omega}_i = -m_i \dot{\hat{P}}_i$ and by \eqref{eq:filter} we have that $\dot{\omega}_i = \frac{m_i}{\tau_i} (\hat{P}_i -P_i)$. Now, we can substitute $\hat{P}_i$ using \eqref{eq:freq} and obtain the frequency control dynamics as
\begin{equation}
\label{eq:droopdyn}
\begin{aligned}
\dot{\theta}_i & = \omega_i \\
\tau_i \dot{\omega}_i & = -\omega_i + \omega^* - m_i(P_i - P_i^*).
\end{aligned}
\end{equation} 

\begin{rem} The second-order frequency droop control model \eqref{eq:droopdyn} for inverter-based networks is analogous to the classical machine model for synchronous generators. The models are equivalent with respect to the performance measure considered here, see \cite{Tegling2014, Tegling2015}. We regard the parameters {$\tau_i, m_i$} in \eqref{eq:droopdyn} as design parameters and thus assume inverter-based networks throughout, although the setting can easily be extended to networks with both inverters and synchronous generators. We refer to \cite{Andreasson2014ACC} for a discussion on distributed PI control in synchronous generator networks. 
\end{rem}

\subsection{Distributed averaging proportional integral (DAPI) controller}
\label{sec:dapi}
 The droop controller \eqref{eq:droop} is completely decentralized, requiring only local measurements of active power for implementation. Under reasonable conditions, \eqref{eq:droop} guarantees the desired power sharing, and synchronizes the inverter network to a common steady-state frequency $\omega_{\rm ss}$; see \cite{SimpsonPorco2013} for an analysis. However, as droop control is effectively proportional control, it typically leads to static deviations of the steady-state frequency $\omega_{\rm ss}$ from the nominal frequency $\omega^*$. This deficiency motivates so-called secondary integral control, the goal of which is to eliminate the static error. 
 
Following \cite{SimpsonPorco2013, Andreasson2014ACC}, in this paper we consider a distributed integral control strategy which we refer to as \emph{distributed averaging proportional integral (DAPI)} control. For this purpose, assume that the inverters in the physical network, as described by $\mathcal{G} = \{\mathcal{V}, \mathcal{E}\}$, have access to a communication network represented by the graph $\mathcal{G}^C = \{\mathcal{V}, \mathcal{E}^C\}$. Let $\mathcal{N}^{C}_i$ denote the neighbor set of node $i$ in $\mathcal{G}^C $. The controller takes the form 
\begin{subequations}
\label{eq:dapidyn}
\begin{align} \label{eq:thetadapi}
\dot{\theta}_i & = \omega_i \\ \label{eq:omegadapi}
\tau_i \dot{\omega}_i & = -\omega_i + \omega^* - m_i(P_i - P_i^*) + \Omega_i \\ \label{eq:Omegadapi}
k_i \dot{\Omega}_i & = -\omega_i + \omega^* - \sum_{j  \in \mathcal{N}^{C}_i } c_{ij}(\Omega_i - \Omega_j),
\end{align}
\end{subequations}
where $\Omega_i$ is the secondary control variable and $k_i>0$ and $c_{ij} = c_{ji} >0, ~i \in \mathcal{V}, j \in \mathcal{N}^{C}_i$ are controller parameters. Notice that equations \eqref{eq:thetadapi}~-~\eqref{eq:omegadapi} are the droop controller dynamics \eqref{eq:droopdyn}, but with the additional secondary control input $\Omega_i$. { Hence, \eqref{eq:Omegadapi} can be thought of as a distributed integral controller appended to \eqref{eq:thetadapi}~-~\eqref{eq:omegadapi}.}

As shown in \cite{SimpsonPorco2013}, if the communication network $\mathcal{G}^{C}$ among the inverters is connected, the distributed controller \eqref{eq:dapidyn} restores the network frequency to $\omega^*$ while maintaining an optimal steady-state distribution of power injections among the inverters established by droop control. When all gains $c_{ij}$ are zero, \eqref{eq:Omegadapi} degenerates into a decentralized integral controller, and in this case \eqref{eq:dapidyn} possesses a large subspace of undesirable equilibria \cite[Lemma 4.1]{Dorfler2014}. In practice, such a control design destabilizes the network unless the controllers have access to accurate phasor measurements units (PMUs). We refer to \cite{Andreasson2014TAC} for an elaboration. 

To simplify upcoming notation, we define the weighted Laplacian matrix $L_C \in \mathbb{R}^{N \times N}$ of the communication graph $\mathcal{G}^C $ by ($L_{C,ij}$ denotes the element at row $i$ and column $j$):
\begin{equation}
\label{eq:Lcdef}
L_{C,ij} = \begin{cases} -c_{ij} & \text{if }j \in \mathcal{N}^{C}_i, ~j \neq i \\
\sum_{k \in \mathcal{N}^{C}_i} c_{ik } & \text{if } j = i\\
0 & \text{otherwise} \end{cases}
\end{equation} 

\begin{rem}
\label{rem:1storder}
The models \eqref{eq:droopdyn} and \eqref{eq:dapidyn} reduce to the first-order inverter models considered in \cite{SimpsonPorco2015} if $\tau_i = 0$ for all $i \in \mathcal{V}$.
\end{rem}

\subsection{Power flow}
\label{sec:powerflow}
Introducing $\theta_{ij} = (\theta_i - \theta_j)$ as the phase angle difference between neighboring nodes, we can write the active power injected to the grid at node $i\in \mathcal{V}$ as
\begin{equation}
\label{eq:powerflow}
P_i = \bar{g}_i |V_i|^2 + \sum_{j \in \mathcal{N}_i} |V_i||V_j|(g_{ij}\cos\theta_{ij} + b_{ij}\sin \theta_{ij} ).
\end{equation}
Here, $\mathcal{N}_i$ denotes the neighbor set of node $i$ in $\mathcal{G}$. $g_{ij}$ and $b_{ij}$ are respectively the conductance and susceptance associated with the line $e_{ij}$, and $\bar{g}_i$ is the shunt conductance of node $i$. As per convention in power flow analysis, we assume that all quantities in \eqref{eq:powerflow} have been normalized by system constants and are measured in per unit (p.u.). 


In what follows, we will use a simplified model in which we consider small deviations from a stable operating point. We can therefore approximate the power flows using the standard linear power flow assumption:
\begin{equation}
\label{eq:linprwflow}
P_i \approx \sum_{j \in \mathcal{N}_i} b_{ij} (\theta_i - \theta_j ) .
\end{equation}
See e.g. \cite{Purchala2005} for a general analysis of the applicability of this assumption and \cite{Tegling2015} for an error estimate with respect to the performance measure of interest.

In upcoming notation, we will use the network admittance matrix $Y\in \mathbb{C}^{N \times N}$, with elements given by $Y_{ii} = \bar{g}_i +\sum_{k \in \mathcal{N}_i} y_{ik}$, $Y_{ij} = -y_{ij}$ if $j \in \mathcal{N}_i, ~j \neq i$ and zero otherwise. 
The matrix $Y$ can be partitioned into a real and an imaginary part:
\begin{equation} \label{eq:defY}
Y = L_G + \mathrm{diag}\{\bar{g}\} - \mathbf{j}(L_B) , 
\end{equation} 
where $L_G$ denotes the network's conductance matrix and $L_B$ its susceptance matrix. By definition, the matrices $L_B$ and $L_G$ are weighted graph Laplacians of $\mathcal{G}$, with edge weights respectively defined by $b_{ij}$ and $g_{ij}$.  

Substituting the power flow equation \eqref{eq:linprwflow} into, respectively, the dynamics \eqref{eq:droopdyn} and \eqref{eq:dapidyn}, we notice that an equilibrium is given by $\omega = \omega^*$, $\theta = L_B^\dagger P^*$ and $\Omega = 0$ ($\dagger$ denotes the Moore-Penrose pseudo inverse). Without loss of generality, we translate this operating point to the origin through a change of variables. 

Further, we assume that the system is subject to small disturbances or persistent small amplitude noise, representing e.g. generation and load fluctuations, which we model as a distributed disturbance input $w$ acting on the inverters. We can then summarize the system dynamics as follows:
\vspace{2pt}


\noindent \textit{Standard droop control:}
 \begin{align}
\label{eq:ssdroop}
\begin{bmatrix}
\dot{\theta} \\ \dot{\omega} 
\end{bmatrix}& = \begin{bmatrix}
0 & I \\ -MT^{-1}L_B & -T^{-1}
\end{bmatrix} \begin{bmatrix}
\theta \\ \omega
\end{bmatrix} +\begin{bmatrix}
0 \\ T^{-1}
\end{bmatrix}w \\ \nonumber
& =: A_{\mathrm{std}} \psi_{\mathrm{std}} + B_{\mathrm{std}}w,
\end{align}

\noindent \textit{DAPI control:}
\begin{align} \nonumber 
\begin{bmatrix}
\dot{\theta} \\ \dot{\omega} \\ \dot{\Omega}
\end{bmatrix} = &\begin{bmatrix}
0 & I & 0 \\
-MT^{-1}L_B & -T^{-1} & T^{-1}\\
0 & -K^{-1} & -K^{-1}L_C
\end{bmatrix} \begin{bmatrix}
\theta \\ \omega \\ \Omega
\end{bmatrix} \\ \label{eq:ssdapi}  &  + \begin{bmatrix}
0 \\ T^{-1} \\ 0 
\end{bmatrix}w  =: A_{\mathrm{DAPI}} \psi_{\mathrm{DAPI}} + B_{\mathrm{DAPI}}w.
\end{align} 
Here, we have introduced the column vectors $\theta, ~\omega, ~\Omega$ containing the translated system states, with total state vectors $\psi_{\rm std} = (\theta,\omega)^T$ and $\psi_{\rm DAPI} = (\theta,\omega,\Omega)^T$. The system parameters are given by $M = \mathrm{diag}\{m_i\}$, $T = \mathrm{diag}\{\tau_i\}$, and  $K = \mathrm{diag}\{k_i\}$.

\subsection{System performance}
In this paper, we are concerned with the performance of the systems \eqref{eq:ssdroop}~-~\eqref{eq:ssdapi} in terms of the resistive power losses incurred in returning the system to a synchronous state following a small transient event, or in maintaining this state under persistent stochastic disturbances $w$. These losses are associated with the power flows that arise from fluctuating phase angle differences, and can be regarded as the control effort required to drive the system to a steady state with desired active power sharing. 

To define the relevant performance measure, we adopt the approach first presented in \cite{BamiehGayme2013}. Consider the real power loss over the edge $e_{ij}$, given by Ohm's law as $P^{\rm loss}_{ij} = g_{ij}|V_i - V_j|^2$. If we enforce the linear power flow assumptions and retain only the terms that are quadratic in the state variables, standard trigonometric methods give that $P^{\rm loss}_{ij} \approx g_{ij}(\theta_i - \theta_j)^2$. Since $\theta_i$ represents deviations from an operating point, this is equivalent to the power loss over the edge during the transient. The total instantaneous losses over the network are then approximately
\begin{equation}
\label{eq:pwrloss}
\mathbf{P}^{\mathrm{loss}} = \sum_{e_{ij} \in \mathcal{E}} g_{ij}(\theta_i - \theta_j)^2,
\end{equation}
which we can write as the quadratic form $\mathbf{P}^{\mathrm{loss}}  = \theta^T L_G \theta$. Since $L_G$ is a positive semidefinite graph Laplacian, it has a unique positive semidefinite square-root $L_G^{1/2}$. We can therefore define outputs of the systems \eqref{eq:ssdroop}~-~\eqref{eq:ssdapi} respectively as 
\begin{align} \label{eq:outputdroop}
y &= \begin{bmatrix}
L_G^{1/2} & 0 
\end{bmatrix} \psi_{\mathrm{std}} =: C_{\mathrm{std}} \psi_{std} \\ \label{eq:outputdapi}
y &= \begin{bmatrix}
L_G^{1/2} & 0  & 0
\end{bmatrix}\psi_{\mathrm{DAPI}} =: C_{\mathrm{DAPI}} \psi_{DAPI}
\end{align} which both give that $\mathbf{P}^{\mathrm{loss}} = y^Ty$. We now have two input-output mappings from $w$ to $y$: $H_{\mathrm{std}}$ given by \eqref{eq:droopdyn}, \eqref{eq:outputdroop} and $H_{\mathrm{DAPI}}$ given by \eqref{eq:dapidyn}, \eqref{eq:outputdapi}, which are linear-quadratic approximations of the full nonlinear problems. 

We have just established that the  instantaneous resistive losses incurred in the transient can be approximated by the (squared) Euclidean norm of the output $y$. 
 The losses due to a white noise disturbance input can thus be evaluated as the system's \hn norm, which is  
\[||H||_2^2 = \lim_{t\rightarrow \infty }\mathbb{E}\{y^T(t)y(t)\}. \]
The use of the \hn norm to quantify power losses can also be motivated under other input scenarios, see \cite{Tegling2014}.
\begin{rem}
The systems \eqref{eq:ssdroop}~-~\eqref{eq:ssdapi} represent linearized control dynamics in which line resistances are not present in the first approximation, having been assumed small compared to the line reactances. The outputs \eqref{eq:outputdroop}~-~\eqref{eq:outputdapi} represent quadratic approximations of the power losses and measure the effect of non-zero line resistances, given the state trajectories arising from the control dynamics. A rigorous justification for these assumptions are given in \cite{Tegling2015}.
\end{rem}

\section{Performance analysis}
\label{sec:eval}
In this section, we derive closed-form expressions for the performance of the systems \eqref{eq:ssdroop}~-~\eqref{eq:ssdapi} with respect to the outputs \eqref{eq:outputdroop}~-~\eqref{eq:outputdapi}, 
%
under the following Assumptions:
\begin{enumerate}[(i)]
\item \textit{Identical inverters.} All inverters have identical parameter settings and low-pass filters, i.e., $M = \mathrm{diag}\{m\}$, $T = \mathrm{diag}\{\tau\}$, $K = \mathrm{diag}\{k\}$.
\item \emph{Uniform resistance-to-reactance ratios.} The ratio of resistance to reactance, equivalently conductance to susceptance, of all lines are uniform and constant, i.e., 
\begin{equation}
\label{eq:alphadef}
\alpha : = \frac{g_{ij}}{b_{ij}}, 
\end{equation}
for all $e_{ij} \in \mathcal{E}$. This implies $L_G= \alpha L_B$.
\item \emph{Communication network topology.} The topology of the communication network $\mathcal{G}^{C}$ is identical to that of the physical network $\mathcal{G}$. We also assume 
\begin{equation}
\label{eq:gammadef}
L_C = \gamma L_B,
\end{equation}
i.e., $\gamma = \frac{c_{ij}}{b_{ij}}$, with $\gamma\ge 0$, for all $e_{ij}\in \mathcal{E} = \mathcal{E}^C$.
\end{enumerate}
Assumption (ii), which is also applied in e.g. \cite{LuChu2013, Dorfler2014}, can be motivated first by a uniformity in the physical line properties in a microgrid (i.e., materials and dimensions). Kron reduction of a network also increases its uniformity in node degrees \cite{Motter2013}. This makes the line properties more uniform in an effective network model.

Assumption (iii) implies that the secondary control layer is set up along the physical network lines, and is shown in \cite{Andreasson2014ACC} to constitute a sufficient criterion for load sharing with minimized generation costs. The assumption (17) says that the gain on the averaging term $\Omega_i-\Omega_j$ is set in proportion to the line susceptance $b_{ij}$, and will help us to obtain explicit analytic expressions for the \hn norms. In Section~\ref{sec:discussion}, we discuss possible implications of a relaxation of that assumption. 



\subsection{Input-output analysis}
The susceptance matrix $L_B$ is a weighted graph Laplacian and as such, it has a well-known eigenvalue at zero with the associated eigenvector $\mathbf{1} = (1,1,\ldots,1)^T$, that is, $L_B\mathbf{1} = 0$. The system matrices $A_{\mathrm{std}}$ in \eqref{eq:ssdroop} and $A_{\mathrm{DAPI}}$ in \eqref{eq:ssdapi} inherit this zero eigenvalue, which corresponds to a uniform drift of all phase angles $\theta$. This mode is, however, unobservable from the outputs, keeping the systems $H_{\rm std}$ and $H_{\rm DAPI}$ input-output stable under the given assumptions. 

The derivation of our main result relies on a unitary state transformation that divides the systems $H_{\rm std}$ and $H_{\rm DAPI}$ into $N$ decoupled subsystems, each associated with an eigenvalue $\lambda_n$ of $L_B$, for $n = 1,\ldots, N$. The \hn norm of the subsystem corresponding to the zero mode vanishes. Therefore, the full system's squared \hn norm becomes the sum of that of $N-1$ subsystems. The details of the derivation are outlined in the appendix.
\begin{theorem} \label{lem:mainres} Under Assumptions (i)~-~(iii), 
the squared \hn norm of the input-output mapping $H_\mathrm{std}$ is
\begin{equation}
\boxed{
\label{eq:stdnorm}
||H_{\mathrm{std}}||_2^2 = \frac{\alpha}{2m}(N-1).}
\end{equation}
The corresponding norm of the mapping $H_\mathrm{DAPI}$ is
\begin{equation}
\boxed{
\label{eq:dapinorm}
||H_{\mathrm{DAPI}}||_2^2 =  \frac{\alpha}{2m} \sum_{n = 2}^N \frac{1}{1+ \frac{\gamma \tau \lambda_n +k}{ \gamma \lambda_n (\gamma \tau \lambda_n +k) + k^2m \lambda_n }}.}
\end{equation}
These expressions represent the expected power losses due to a white noise disturbance input $w$.
\end{theorem}
\begin{proof}
See appendix.
\end{proof}

The result in \eqref{eq:stdnorm} is the same as was obtained for networks of synchronous generators in \cite{BamiehGayme2013}, identifying the droop coefficient $m$ with the generator damping. 
These power losses scale linearly with the number of nodes $N$, a fact which seems to be a fundamental limitation to performance in networks where power flows are the mechanism by which the system regulates frequency, as proposed in \cite{Tegling2014}. When a secondary control layer is added through the DAPI control, the losses \eqref{eq:dapinorm} still grow with the number of nodes, but they are smaller in absolute terms. Consider the following Corollary to Theorem~\ref{lem:mainres}:
\begin{corollary}
For all $m,k, \tau, \gamma >0$, 
\[
||H_{\mathrm{DAPI}}||_2^2 < ||H_{\mathrm{std}}||_2^2 ,\]
i.e., the expected power losses due to the disturbance $w$ are smaller with the DAPI control strategy than with the standard droop control.
\end{corollary}
\begin{proof}
Notice that $1+ \frac{\gamma \tau \lambda_n +k}{ \gamma \lambda_n (\gamma \tau \lambda_n +k) + k^2m \lambda_n}>1$, since all terms are positive. Hence, $||H_{\mathrm{DAPI}}||_2^2< \frac{\alpha}{2m} \sum_{n = 2}^N 1 = \frac{\alpha}{2m}(N-1) = ||H_{\mathrm{std}}||_2^2 $.
\end{proof}

We note that the norms \eqref{eq:stdnorm} and \eqref{eq:dapinorm} both scale linearly with network's resistance-to-reactance ratio $\alpha$, and hence that the ratio of the norms is independent of $\alpha$. This suggests that, to first order, the relative performance improvement of distributed PI-control over droop control does not deteriorate as grid resistances increase. 
\begin{figure}
\centering
\begin{tikzpicture}
	\pgfplotsset{every tick label/.append style={font=\footnotesize}}
	\begin{axis}
	[xlabel={$N$ },
	ylabel={$||H||_2^2$ },
	ylabel near ticks,
	xlabel near ticks,
	xmin=0,
	xmax=100,
	ymin=0,
	ymax=10,
	yticklabel style={/pgf/number format/.cd,
		fixed,
		precision=4},
	grid=major,
	height=4.5cm,
	width=\columnwidth,
	legend cell align=left,
	legend style={at={(0,1)},anchor= north west, font = \scriptsize},
	legend entries={Standard droop control\\  DAPI, complete graph \\ DAPI, line graph \\}, cycle list name=mygraylist,
	]
	\foreach \x in {1, 2,3}{
	\addplot +[mark=none, very thick, ] table[x index=0,y index=\x,col sep=space]{Figures/Plot_data/linecomplete.txt};
	}	
	\end{axis}
	\end{tikzpicture}	
	\vspace{-0.45cm}
\caption{\hn norms in \eqref{eq:stdnorm}~-~\eqref{eq:dapinorm} for sample networks of size $N$ with line graph and complete graph topologies. Note that $||H_{\rm std}||_2^2$ in \eqref{eq:stdnorm} is topology-independent. Here, $k = \gamma = m = 1$, line susceptances $b_{ij}$ are uniformly distributed on $[0.5 , 1.5]$.  }
\label{fig:linecomplete}
\end{figure}
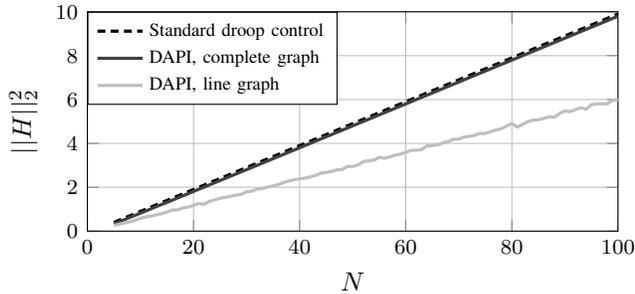

\subsection{Losses' dependence on network connectivity}
It is interesting to note that while the losses under standard droop control \eqref{eq:stdnorm} are entirely independent of network topology, the losses that are incurred under DAPI control \eqref{eq:dapinorm} depend on network topology through the eigenvalues $\lambda_n$ of $L_B$. In fact, the expression is monotonically increasing in $\lambda_n$, implying that losses grow with increasing network connectivity. {This in particular implies that the relative performance improvement of DAPI control over droop control will be largest for sparse network topologies, such as those found in standard distribution networks and microgrids.} The best performance can be expected to be achieved for a line graph topology. In Fig.~\ref{fig:linecomplete} we compare such a topology to a complete graph with respect to the results in Theorem~\ref{lem:mainres}. Although losses for both topologies grow with the network size, as discussed in the previous section, the comparison confirms the lower losses obtained in the line graph case.  

The fact that a loosely interconnected network may outperform a highly interconnected network by incurring smaller power losses in maintaining synchrony is surprising in light of typical notions of power system stability. For example, the connectivity of a network is directly related to its ability to synchronize \cite{Dorfler2010,Pecora1998,Schiffer2014} as well as its damping and rate of convergence \cite{Tang2011}. Our results show that, although additional network lines may improve phase coherence and stability, they also lead to additional power flows that incur losses. Hence, there is a trade-off between performance objectives.


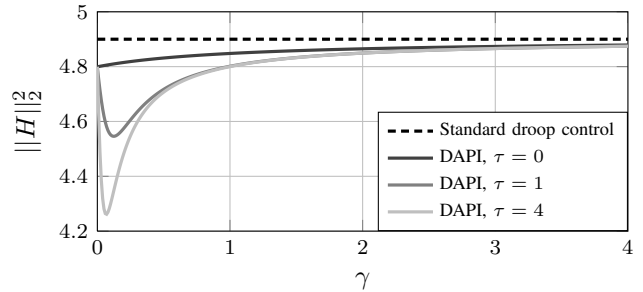
\begin{figure}
\centering
\begin{tikzpicture}
	\pgfplotsset{every tick label/.append style={font=\footnotesize}}
	\begin{axis}
	[xlabel={$\gamma$ },
	ylabel={$||H||_2^2$ },
	ylabel near ticks,
	xlabel near ticks,
	xmin=0,
	xmax=4,
	ymin=4.2,
	ymax=5,
	yticklabel style={/pgf/number format/.cd,
		fixed,
		precision=4},
	grid=major,
	height=4.5cm,
	width=\columnwidth,
	legend cell align=left,
	legend style={at={(1,0)},anchor= south east, font = \scriptsize},
	legend entries={Standard droop control\\  DAPI, $\tau = 0$ \\ DAPI, $\tau = 1$ \\ DAPI, $\tau = 4$ \\}, 
	cycle list name=myothergraylist
	]
	\foreach \x in {1,2,3,4}{
	\addplot +[mark=none, very thick, ] table[x index=0,y index=\x,col sep=space]{Figures/Plot_data/varygamma.txt};	}	
	\end{axis}
	\end{tikzpicture}	
\caption{\hn norms in \eqref{eq:dapinorm} as a function of $\gamma$ for a complete graph with $N = 50$ nodes. Here, $k = m = 1$, and the filter time constant $\tau \in \{0,1,4\}$. For $\tau = 0$, the system \eqref{eq:ssdapi} reduces to a first order model, and the optimal $\gamma^* = 0$.   }
\label{fig:gamma}
\end{figure}

\begin{figure*}[htb]
\centering
\includegraphics[width = 0.85\textwidth]{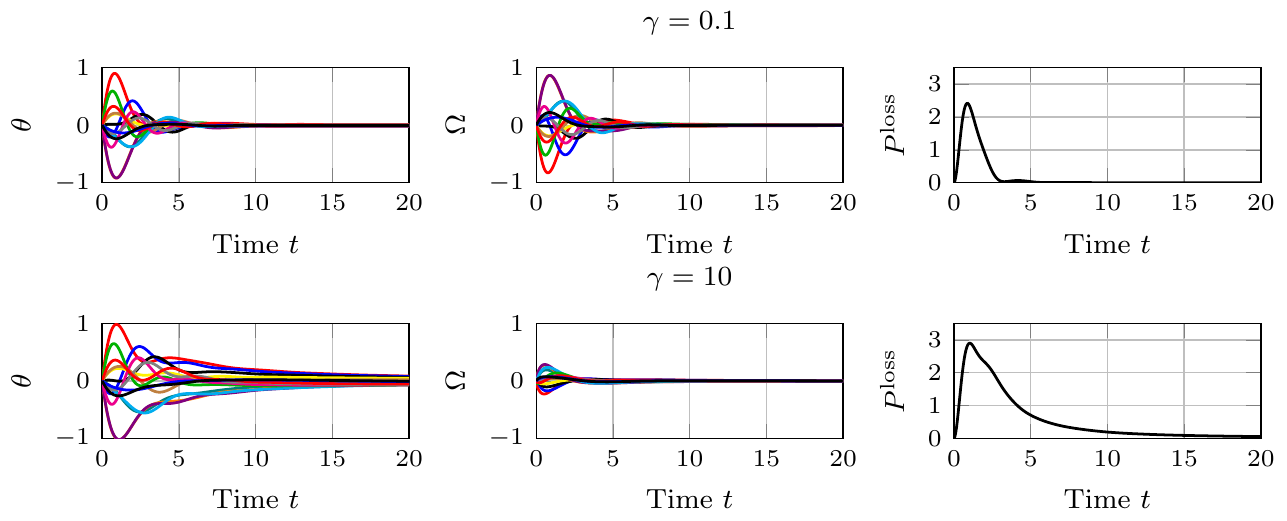}
\caption{Simulation of the system \eqref{eq:ssdapi} on a 20-node line network, with associated power losses \eqref{eq:pwrloss}. Here, $m = k = \tau = 1$ and $\gamma = 0.1$ (upper panel) and $\gamma = 10$ (lower panel). }
\label{fig:simulation}
\end{figure*}

\section{Control design for loss reduction }
\label{sec:analysis}
In the previous section, we established that the DAPI control strategy improves performance in terms of transient power losses for droop-controlled microgrids. We now turn to the question of optimal tuning of this controller. That is, how should the integral action $k$ in \eqref{eq:dapidyn} and the communication gain parameter $\gamma$ in \eqref{eq:gammadef} be chosen to minimize transient losses, with respect to a given droop-controlled network.  


\subsection{Communication gain}
As discussed in Section~\ref{sec:setup}, distributed PI control requires a communication network through which inverters can communicate their secondary control variables $\Omega_i$. 
While any non-zero gains $c_{ij}$ for the distributed averaging will guarantee that the control objectives are reached \cite{SimpsonPorco2015}, an important design question is how to choose these gains to optimize {the transient performance considered herein}. In our case, this choice is reflected through the parameter $\gamma$ in \eqref{eq:gammadef}.

Fig.~\ref{fig:gamma} displays the transient power losses associated with the DAPI control, as given by \eqref{eq:dapinorm}, as a function of $\gamma$ for a sample network with a complete graph structure. As the figure indicates, it turns out that there exists a distinct optimal value for $\gamma \ge 0$:
\begin{lemma}
\label{lem:optimum}
For a given network with DAPI control \eqref{eq:ssdapi} and under Assumptions (i)~-~(iii) of Section~\ref{sec:eval}, there is a unique communication gain ratio $\gamma^* $ which minimizes the \hn norm \eqref{eq:dapinorm}. 
\end{lemma}
\begin{proof}
The optimum is given by the positive root of the equation $\frac{\rm d}{\mathrm{d}\gamma}||H_{\rm DAPI}||_2^2 \big|_{\gamma = \gamma^*} = 0$.  If there is no such root, $\gamma^*=0$. The details are omitted due to space limitations.  
\end{proof}
The value of $\gamma^*$ is strongly dependent on the network parameters, but once these are given, it is easy to find the optimal tuning. We note that the optimal $\gamma^*$ is often very small, in particular if the time constant $\tau$ is small. In the limit where $\tau = 0$, we have $\gamma^* =0$. However, we cannot choose a design where $\gamma = 0$ without causing an undesirable drift in the system, which in practice causes instabilities (see Section~\ref{sec:dapi}). If $\gamma$, on the other hand, is set too large, the distributed averaging term of \eqref{eq:dapidyn} converges too fast compared to the phase angles, and deteriorates the damping effect of the secondary control. A simulation of this case is shown in Fig.~\ref{fig:simulation}.

For complete graphs, the potential for performance improvement is smaller than for more sparsely connected networks. An optimized controller tuning is therefore particularly relevant. For this case, we provide a closed-form expression for $\gamma^*$:
\begin{corollary}
\label{lem:completeoptimum}
If the graph underlying the network $\mathcal{G}$ is complete and the line susceptances $b_{ij} = b$ for all $e_{ij} \in \mathcal{E}$, then $\gamma^*$ is given by
\begin{equation}
\label{eq:optimalcomplete}
\gamma^* = \frac{k}{Nb\tau}\left( \sqrt{Nb m \tau } -1 \right)
\end{equation}
if $Nb m \tau>1$. Otherwise, $\gamma^* = 0$
\end{corollary}
\begin{proof}
When edge weights $b$ are uniform, the $N-1$ non-zero eigenvalues of the complete graph Laplacian $L_B \in \mathbb{R}^{N\times N}$ are all given by $Nb$. It then suffices to evaluate $\frac{\mathrm{d}}{\mathrm{d}\gamma}\frac{1}{1+ \frac{\gamma \tau Nb +k}{ \gamma Nb (\gamma \tau Nb +k) + k^2m Nb}} = 0 $ and the result follows. 
\end{proof}

\subsection{Integral action}
Now, consider the parameter $k$ in \eqref{eq:dapidyn}, which reflects the amount of integral action in the DAPI controller. First, notice that in the limit where $k \rightarrow \infty$ the integral action vanishes and the standard droop control dynamics \eqref{eq:droopdyn} are retrieved, with the associated \hn norm \eqref{eq:stdnorm}. It is easy to show based on \eqref{eq:dapinorm} that as $k$ then decreases, losses are reduced monotonically and at an increasing rate. On the other hand, in the theoretical limit of an infinitely large integral gain ($k = 0$), the system can become arbitrarily well damped and losses minimized. Fig.~\ref{fig:kgamma} displays the relative performance improvement achieved through the DAPI strategy as a function of $k$, for a hypothetical network based on the IEEE 57-bus benchmark system topology \cite{testarchive}.

Our results also indicate that the importance of the distributed averaging term in \eqref{eq:dapidyn} increases as the integral action decreases. That is, the optimal communication gain given by $\gamma^*$ grows as $k$ grows. For a complete graph with uniform edge weights, this relationship is linear, by Corollary~\ref{lem:completeoptimum}. For the IEEE 57-bus benchmark system topology we display this relationship between $k$ and $\gamma^*$ in Fig.~4.

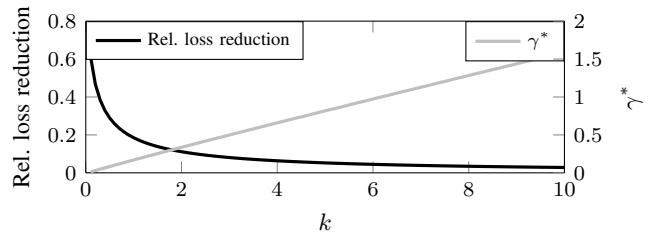
\begin{figure}
\centering

	\begin{tikzpicture}
	\pgfplotsset{every tick label/.append style={font=\footnotesize}}
		\pgfplotsset{every axis title/.append style={font=\small}}
		\pgfplotsset{every axis label/.append style={font=\small}}
\begin{axis}
	[xlabel={$k$ },
	ylabel={Rel. loss reduction },
	ylabel near ticks,
	xlabel near ticks,
	xmin=0, xmax=10,
	ymin=0, ymax=0.8,
	axis y line*= left,
	height=3.6cm,
	width=0.92\columnwidth, 
legend cell align=left,
	legend style={at={(0,1)},anchor= north west, font = \scriptsize},
	legend entries={Rel. loss reduction\\   }, 
	]
	\addplot +[mark=none, very thick,color = black ] table[x index=0,y index=1,col sep=space]{Figures/Plot_data/krelloss.txt};	
	\end{axis}
	 \begin{axis}[
    xmin = 0, xmax = 10,
    ymin = 0, ymax = 2,
    height=3.6cm,
    width=0.92\columnwidth,
    hide x axis,
    hide y axis, 
    			legend cell align=left,
	legend style={at={(1,1)},anchor= north east, font = \scriptsize},
	legend entries={$\gamma^*$\\ }, 
	]
  \addplot +[mark=none, very thick, color = gray3 ] table[x index=0,y index=1,col sep=space]{Figures/Plot_data/kgammastar.txt};	
  \end{axis}
		\begin{axis}[	
	xmin=0, xmax=10,
	ymin=0, ymax=2,	
	hide x axis,
	axis y line*= right,
	ylabel={$\gamma^*$ },
	ylabel near ticks,
	height=3.6cm,
	width=0.92\columnwidth
	]
	\end{axis}
	\end{tikzpicture}	
\vspace{-0.5cm}
\caption{Relative loss reduction with DAPI control for a test network based on the IEEE 57 bus benchmark system topology, at $\gamma = \gamma^*$, as function of $k$. Here, $m = \tau = 1$.   }
\label{fig:kgamma}
\end{figure}

\section{Discussion}
\label{sec:discussion}
In this paper, we have evaluated transient performance of an inverter-based microgrid in terms of the power losses incurred in regulating the frequency to a synchronous state after a disturbance, or in maintaining this state under persistent small disturbances. We compared two control strategies: the standard frequency droop controller and a distributed averaging PI (DAPI) controller and found that the latter has the potential to significantly reduce the transient power losses. This relative performance improvement compared to droop control is largest for sparse network topologies, such as those found in standard distribution networks and microgrids.  

This result is in sharp contrast both to previous results in \cite{BamiehGayme2013,Sjodin2014,Tegling2014}, where losses associated with frequency regulation were shown to be independent of network connectivity, as well as to standard notions of power system stability, which typically predict highly interconnected networks to have better performance. The apparent reason for our results is the self-damping terms $-\omega_i + \omega^*$ added to the consensus dynamics in \eqref{eq:Omegadapi}. These terms attenuate disturbances independently of the power flows. Increasing connectivity by introducing more lines generates more power flows, which do not affect the self damping, but increase losses.  

It is important to note, however, that the losses' scaling with the size of the network remains unchanged by the DAPI strategy, and seems to be a fundamental performance limit in systems where active power flows are the mechanism by which the system regulates frequency. Therefore, even though transient power losses typically represent a small percentage of the  total power flow, our results indicate that they may become significant when power networks become increasingly distributed and the number of generators grows. Since DAPI control both reduces transient losses and eliminates control errors, our results provide additional arguments in favor of distributed algorithms for secondary frequency control in microgrids.

We also derived results on optimal tuning of the DAPI controller for loss reduction. In particular, the distributed averaging term of \eqref{eq:dapidyn}: $\sum_{j \in \mathcal{N}^C_i} c_{ij}(\Omega_i - \Omega_j)$ should be tuned so that $c_{ij} = \gamma^*b_{ij}$, where $b_{ij}$ is the line susceptance and $\gamma^*$ is a unique positive optimizer. Too large communication gains $c_{ij}$ cause a too strong reliance on the distributed averaging in relation to the integral control, which deteriorates damping and increases losses. In the present work, we made the restrictive assumption that the graph topology for the distributed averaging follows that of the physical network, and found controller tunings that minimize losses. With more degrees of freedom, we conjecture that losses can be even further reduced by judicious control design. An important direction for future work is therefore to find an optimal topology configuration of the communication network.

\section{Acknowledgements}
We would like to thank Bassam Bamieh (UCSB) and Dennice Gayme (the Johns Hopkins University) for their many insightful comments and several interesting discussions. Funding support from the Swedish Research Council through grant 2013-5523 and the Swedish Foundation for Strategic Research through the project ICT-Psi is also gratefully acknowledged.
\section*{Appendix}
\subsection{Proof of Theorem~\ref{lem:mainres} }
We follow the approach in \cite{BamiehGayme2013} and transform the state vectors of $H_{\mathrm{std}}$ and $H_{\mathrm{DAPI}}$ so that: $\theta = U\hat{\theta}$, $\omega = U\hat{\omega}$ and $\Omega = U \hat{\Omega}$. Let $U$ be the unitary matrix which diagonalizes $L_B$. That is, $L_B = U^*\Lambda_B U$ with  $\Lambda_B = \mathrm{diag}\{\lambda_1, \lambda_2, \ldots, \lambda_N\}$. Note that by assumptions (ii)~-~(iii), $L_B,~L_C$ and $L_G^{1/2}$ are simultaneously diagonalizable, so $U^*L_CU = \gamma \Lambda_B$ and $U^*L_G^{1/2}U = \sqrt{\alpha} \Lambda_B^{1/2}$. Given that the \hn norm is unitarily invariant, we can also define $\hat{y} = U^*y$ and $\hat{w} = U^*w$. 

Through these transformations, we obtain the systems $\hat{H}_{\mathrm{std}}$ and $\hat{H}_{\mathrm{DAPI}}$ in which all blocks have been diagonalized. They thus each represent $N$ decoupled subsystems:

\noindent \textit{Standard droop control:}
\begin{align} \nonumber
\begin{bmatrix}
\dot{\hat{\theta}}_n \\ \dot{\hat{\omega}}_n 
\end{bmatrix} =& \begin{bmatrix}
0 & 1 \\ -\frac{m}{\tau} \lambda_n & -\frac{1}{\tau}
\end{bmatrix} \begin{bmatrix}
\hat{\theta}_n \\ \hat{\omega}_n
\end{bmatrix} +\begin{bmatrix}
0 \\ \frac{1}{\tau}
\end{bmatrix}\hat{w}_n \\ \label{eq:subsysdroop}
& =: \hat{A}_{\mathrm{std},n} \hat{\psi}_{\mathrm{std},n} + \hat{B}_{\mathrm{std},n}\hat{w}_n \\
\nonumber  \hat{y}_n =& \sqrt{\alpha \lambda_n} \begin{bmatrix}
1 & 0 
\end{bmatrix} \begin{bmatrix}
\hat{\theta}_n \\ \hat{\omega}_n
\end{bmatrix} = : \hat{C}_{\mathrm{std},n} \hat{\psi}_{\mathrm{std},n}. 
\end{align}

\noindent \textit{DAPI control:}
\begin{align} \nonumber
\begin{bmatrix}
\dot{\hat{\theta}}_n \\ \dot{\hat{\omega}}_n \\ \dot{\hat{\Omega}}_n
\end{bmatrix} =& \begin{bmatrix}
0 & 1 & 0 \\ -\frac{m}{\tau} \lambda_n & -\frac{1}{\tau} & -\frac{1}{\tau}\\ 
0 & -\frac{1}{k} & - \frac{1}{k}\gamma \lambda_n
\end{bmatrix} \begin{bmatrix}
\hat{\theta}_n \\ \hat{\omega}_n\\\hat{\Omega}_n
\end{bmatrix} +\begin{bmatrix}
0 \\ \frac{1}{\tau} \\ 0
\end{bmatrix}\hat{w}_n \\ \label{eq:subsysdapi}
& =: \hat{A}_{\mathrm{DAPI},n} \hat{\psi}_{\mathrm{DAPI},n} + \hat{B}_{\mathrm{DAPI},n}\hat{w}_n \\
\nonumber  \hat{y}_n =& \sqrt{\alpha \lambda_n} \begin{bmatrix}
1 & 0 & 0 
\end{bmatrix} \begin{bmatrix}
\hat{\theta}_n \\ \hat{\omega}_n \\ \hat{\Omega}_n
\end{bmatrix} = : \hat{C}_{\mathrm{std},n} \hat{\psi}_{\mathrm{std},n}. 
\end{align}

To verify that all system eigenvalues are in the left half of the complex plane with the exception of the zero eigenvalue associated with $\lambda_1 = 0$,  it suffices to consider the characteristic equations of $\hat{A}_{\rm{std},n}$ and $\hat{A}_{\rm{DAPI},n}$ respectively. The $2N-1$ nonzero eigenvalues $z$ of $A_{\rm std}$ are given by:
\[z^2 + \frac{1}{\tau}z + \frac{m}{\tau}\lambda_n = 0,\]
for $n = 2,\ldots, N$. The $3N-1$ nonzero eigenvalues of $A_{\rm DAPI}$ are given by:
\[ z^3 +z^2\left(\frac{1}{\tau} + \frac{\gamma}{k}\right) + z\left(\frac{1}{k\tau}(\gamma +1)+ \frac{m\lambda_n}{\tau} + \frac{m\gamma \lambda_n}{\tau k} \right) = 0,\]
for $n = 2,\ldots,N$. Since $L_B$ is positive semidefinite, $\lambda_n>0$. It is then easy to verify by Routh's criterion that $\mathrm{Re}\{z\}<0$ if $\gamma,m,k,\tau>0$.

Now, denote the input-output mapping of each such subsystem by $\hat{H}_{\mathrm{std},n}$ and $\hat{H}_{\mathrm{DAPI},n}$ respectively. The squared \hn norms of $H_{\mathrm{std}}$ and $H_{\mathrm{DAPI}}$ are then the sum of the squares of the decoupled subsystems' norms, i.e., $||H_{\mathrm{std}}||_2^2 = ||\hat{H}_{\mathrm{std}}||_2^2 = \sum_{n = 1}^N||\hat{H}_{\mathrm{std},n}||_2^2$, $||H_{\mathrm{DAPI}}||_2^2 = ||\hat{H}_{\mathrm{DAPI}}||_2^2 = \sum_{n = 1}^N||\hat{H}_{\mathrm{DAPI},n}||_2^2$

Notice that the subsystems $\hat{H}_{\mathrm{std},1}$ and $\hat{H}_{\mathrm{DAPI},1}$ corresponding to $\lambda_1 = 0$ have the output $\hat{y}_1 = 0$. This verifies that the zero mode is unobservable and $||\hat{H}_{\mathrm{std},1}||_2^2 = ||\hat{H}_{\mathrm{DAPI},1}||_2^2 =0 $. For $n\neq 0$, the subsystem norms are calculated by solving the Lyapunov equation for the observability Gramians $X_n$:
\begin{equation} \label{eq:lyap}
\hat{A}_n^*X_n + X_n\hat{A}_n = -\hat{C}_n^*\hat{C}_n,
\end{equation}
and taking $||\hat{H}_n||^2_2 = \mathrm{tr}\{\hat{B}_nX_n\hat{B}_n\}$. The subscripts for the standard droop control and the DAPI systems have here been left out to indicate that the equations hold for both. 

Due to space limitations, we omit the expansion of \eqref{eq:lyap}, but refer to \cite{BamiehGayme2013} or \cite{Tegling2015} for a similar derivation.

\bibliographystyle{IEEETran}
\bibliography{EmmasBib15}

\end{document}